\documentclass[12pt]{amsart}
\usepackage{amssymb,latexsym, amsmath, amsxtra}
\usepackage{hyperref}
\hypersetup{
    colorlinks=true,
}
\usepackage{graphicx}
\newtheorem{theorem}{Theorem}
\newtheorem{lemma}{Lemma}

\newtheorem{conjecture}{Conjecture}
\newtheorem{proposition}{Proposition}
\newtheorem{remark}{Remark}

\begin{document}

\title{Which numbers are not the sum plus the product of three positive integers?}
\author{Brian Conrey \& Neil Shah}
\address{American Institute of Mathematics, San Jose, CA 95112, USA}
\email{conrey@aimath.org}
\address{Morgan Hill, CA 95037, USA}
\email{neilhemang@gmail.com}
\thanks{This work is partially supported by an NSF grant.}
\date{\today}

\maketitle

\begin{abstract}
    We investigate the number $R_3(n)$ of representations as the sum plus the product of three positive integers. On average, $R_3(n)$ is $\frac{1}{2}\log^2 n$. We give an upper bound for $R_3(n)$ and an upper bound for the number of $n \leq N$ such that $R_3(n) = 0$. We conjecture that $R_3(n)= 0$ infinitely often.
\end{abstract}

\section{Introduction}

It is well-known that the number of representations of an integer as a product of $k$ positive integers, denoted $\tau_k(n)$, satisfies $\tau_k(n) \ll_{\epsilon} n^\epsilon$ for any $\epsilon>0$ as $n$ goes to infinity (see Section \ref{sec:not} for notation). The proof of this standard result relies crucially on the fact that $\tau_k$ is a multiplicative function. For $k=3$ this amounts to counting integral solutions of
$$n=xyz.$$
The slightest perturbation of this problem turns it from an exercise to a seriously hard research question. One such example is the problem of trying to estimate the number of representations of a number $n$ as the sum plus the product of three positive integers, i.e.
\begin{equation} \label{eqn:splusp} n=xyz+x+y+z.\end{equation}

Let $R_3(n)$ denote the number of solutions to  (\ref{eqn:splusp}) in positive integers $x, y, z$. Here are the values of $R_3(n)$ for $1 \leq n \leq 12$:
\begin{eqnarray*}
0, 0, 0, 1, 0, 3, 0, 3, 3, 3, 0, 9
\end{eqnarray*}
Our motivating problem was the following conjecture:
\begin{conjecture}
\label{con:r3bound}
For any $\epsilon >0$ there is a $C(\epsilon)>0$ such that
$$R_3(n)\le C(\epsilon) n^\epsilon$$
for all positive integers $n$.
\end{conjecture}

While we were unable to prove this conjecture, we were able to find and prove a number of other interesting results as well as formulate some new conjectures.

\subsection{Notation}
\label{sec:not}

We use some notation that is standard in analytic number theory but perhaps not standard to everybody. In this section, we define this notation for the reader.

Suppose $f(n)$ and $g(n)$ are functions defined on the positive integers. Then, $f(n) = O(g(n))$ means that there exists some $M > 0$ and $N_0$ (which could depend on $M$) such that $|f(n)| \leq Mg(n)$ for all $n \geq N_0$. Also, $f(n) = \Omega(g(n))$ means that there exists some $M > 0$ and a sequence $n_1 < n_2 < \ldots$ such that $f(n_j) \geq Mg(n_j)$ for all $j$. Additionally, $f(n) \ll g(n)$ is equivalent to $f(n) = O(g(n))$ and $f(n) \gg g(n)$ is equivalent to $g(n) \ll f(n)$. Note that $\ll_a$, $\gg_a$, and $\Omega_a$ indicate that the implicit constant $M$ is dependent on $a$.

Lastly, suppose that $f(n)$ and $g(n)$ tend to $\infty$ with $n$. Then, $f(n) \sim g(n)$ means that $\lim_{n \to \infty} \frac{f(n)}{g(n)} = 1$.

\section{Statement of Results}

We can prove Conjecture \ref{con:r3bound} on average.
\begin{theorem}
As $N$ goes to infinity,
$$\frac{1}{N}\sum_{n\leq N}R_3(n) \sim \frac{1}{2}\log^2N.$$
\end{theorem}
For an upper bound, we can show the following.
\begin{theorem}
\label{thm:repbound}
For any $\epsilon > 0$,
$$R_3(n)\ll_{\epsilon} n^{1/3}\log n(\log\log n)^4.$$
\end{theorem}
In the course of proving an upper bound for $R_3(n)$, we discovered some interesting properties of $R_3(n)$ which we state here.

\begin{proposition}
\label{prop:prime}
If $R_3(n)=0$ for $n > 2$, then $n$ is prime and $2n-3$ is prime.
\end{proposition}

\begin{proof}
Let
$$f_3(x,y,z) = xyz + x + y + z.$$
Observe that $f_3(x,y,1) = (x+1)(y+1)$. Thus, if $n$ is composite, there exists a  triple $(x,y,1)$ such that $f_3(x,y,1) = n$. Thus, $n$ is prime. The second assortion is obviously true for $n=3$. If $n > 3$ and $2n-3$ is composite, then
$$2n-3 = (2x+1)(2y+1)$$
for some positive $x, y$. This implies
$$n = 2xy + x + y + 2,$$
which is a solution with $z = 2$, which means that $f_3(x, y, 2) = n$. Thus, $2n-3$ is prime.
\end{proof}
This result led us to wonder how many primes $p \leq x$ there are for which $R_3(p) = 0$.
Define
$$U_3(N)=\#\{n\le N: R_3(n)=0\}.$$
Using the large sieve we can show the following.
\begin{theorem} \label{thm:lsbound} There exists some constant $c > 0$ such that
$$U_3(N)\ll \frac{N}{e^{c\sqrt{\log N}}}$$
as $N$ goes to infinity.
\end{theorem}
We expect that this estimate is far from optimal. We did calculations for $N$ up to the 250 millionth prime ($5336500537$) and found 2014 prime numbers which are not the sum plus the product of three positive integers.

Based on numerical evidence, we do believe that $U_3(N)\to \infty$ as $N\to \infty$ but very slowly, maybe like a power of $\log N$.

\begin{conjecture}
There exist positive constants $A < B$ such that
$$\log^A N \ll U_3(N) \ll \log^B N.$$
\end{conjecture}

It is interesting to compare this problem
with the analogous problem in four variables. Let $R_4(n)$ denote the number of positive integral solutions to
$$n =xyzw+x+y+z+w$$
and let $$U_4(N) = \#\{n\le N: R_4(n)=0\}.$$
% decide on \ll_epsilon or not
\begin{conjecture} For any $\epsilon >0$,
$$R_4(n)\ll_\epsilon   n^\epsilon.$$
 \end{conjecture}
Again, we can prove this conjecture on average.
\begin{theorem}
As $N$ goes to infinity,
$$\frac{1}{N}\sum_{n \leq x} R_4(n) \sim \frac{1}{6}\log^3N.$$
\end{theorem}
We can also obtain an upper bound of $R_4(n)$.
\begin{theorem}
\label{thm:pwbound}
For any $\epsilon > 0$,
$$R_4(n) \ll_{\epsilon} n^{\frac{1}{2}+\epsilon}.$$
\end{theorem}

By the large sieve inequality we can show the following.
\begin{theorem}
\label{thm:r4thm}
For any $C > 0$ we have
$$U_4(N)\ll \frac{N}{e^{C\sqrt{\log N}}}$$
as $N$ goes to infinity.
\end{theorem}
Unlike the situation for $R_3(n)$, we believe that only a finite number of $n$ have $R_4(n)=0$.
\begin{conjecture}
\label{conj:five}
If $n > 45752$, then
$$R_4(n) > 0.$$
\end{conjecture}

\subsection{Omega results}
If $n \geq 5$, by consideration of solutions where at least one of $x, y, z$ is equal to $1$, it is easy to see that
$$R_3(n)\ge 6\tau(n)-6$$
where $\tau(n) = \tau_2(n)$ is the number of divisors of $n$. Using what we know about how large $\tau(n)$ can get, it follows that for every $\epsilon > 0$,
\begin{equation}
\label{eqn:omega}
    R_3(n)=\Omega_{\epsilon}\left(\exp(\frac{\log (2-\epsilon) \log n}{\log \log n})\right).
\end{equation}
One could also consider a solution where at least one of $x, y, z$ is equal to $2$, in which case we can see
$$R_3(n) \ge 6\tau(n) + 6\tau(2n-3) + O(1).$$
It is not clear to us how big $\tau(n) + \tau(2n-3)$ can be. One could do similarly for $3, 4, \ldots$ and obtain apparently stronger omega results. In particular, let $$\tau(n; q, a) = \sum_{d|n \atop d \equiv a \bmod q} 1.$$
By (\ref{eqn:factor}) (below), we see that
$$R_3(n) \geq \sum_{1 \leq x \leq n^{\frac{1}{3}}} \max(\tau(xn-(x^2-1);x, 1)-2, 0)$$
Note Proposition \ref{prop:prime} is an immediate corollary of this.

We wonder whether
$$R_3(n) = \Omega_{\epsilon}\left(\exp(\frac{\log (3-\epsilon) \log n}{\log \log n})\right),$$
which would be the case if $R_3(n)$ behaves like $\tau_3(n)$.

\subsection{Algorithm for calculating $R_3(n)$}
\label{algorithm}
In order to find all the $n\le 5\times 10^9$ with $R_3(n)=0$,
we need a reasonable algorithm for determining when $R_3(n)>0$. Suppose that $xyz + x + y + z = n$ and that $x \leq y \leq z$. Multiplying through by $x$ and adding $1$ to both sides yields
\begin{equation}
    \label{eqn:factor}
    (xy+1)(xz+1) = nx - x^2 + 1.
\end{equation}
Therefore, if we can find the proper divisors $d$ of $nx-x^2+1$ such that $d \equiv 1\pmod {x}$, then from this set of divisors we can easily determine the complete set of solutions of $f_3(x, y, z) = n$. If there are no proper divisors $d\equiv 1 \bmod x$, then $R_3(n) = 0$. The probabilistic time complexity for prime factorization of a number $m$ is $O(\exp((\log m)^{\frac{1}{2}+\epsilon}))$. We apply this to $m = nx - x^2 + 1$, for all $x \leq n^{\frac{1}{3}}$ and so for each $n$, there are $n^{\frac{1}{3}+\epsilon}$ steps to determine the value of $R_3(n)$.

For the primes $p$ up to $5336500537$, we checked using this method whether or not $R_3(p) = 0$. We computed all $2014$ primes up to $5336500537$ such that $R_3(p)=0$:
\begin{eqnarray*}&&
2, 3, 5, 7, 11, 13, 17, 23, 31, 37, 41, 43, 53, 67, 71, 83, 97, 101, 107, 113,\\
&& \qquad \dots, 5178563387, 5220047297, 5284333573, 5322410117
\end{eqnarray*}
See \href{https://aimath.org/~conrey/data/not_3_representable.txt}{this link} for the full list.

\subsection{Algorithm for calculating $R_4(n)$}

Let $$f_4(x,y,z,w) = xyzw + x + y + z + w.$$
Suppose $f_4(x,y,z,w) = n$ and $x \leq y \leq z \leq w$. Multiplying through by $xy$ and rearranging yields
$$(xyz+1)(xyw+1) = nxy + 1 - x^2y - xy^2.$$
So, to find $R_4(n)$, it is a matter of finding proper divisors $d$ of $nxy + 1 - x^2y - xy^2$ such that $d \equiv 1 \pmod{xy}$. This involves $\ll n^{\frac{1}{2}} \log n$ factorizations of numbers that are less than $n^{\frac32}$. So, the total time to calculate $R_4(n)$ is $O(n^{\frac{1}{2}+\epsilon})$.

% Can potentially do a little betterhttps://www.overleaf.com/project/6063c4252cee2c772a077f7b
Since the number of divisors of $nxy+1-x^2y-xy^2$ is $O((nxy+1-x^2y-xy^2)^{\epsilon}) = O(n^{\epsilon})$, this argument also proves Theorem \ref{thm:pwbound}.

Next, note that $$f_4(x,y,z,1) = f_3(x,y,z) + 1.$$
Therefore, if $R_3(n) > 0$, then $R_4(n+1) > 0$. Thus, to verify that $R_4(n) > 0$ for $45752 < n \leq 5336500538$, it suffices to check that $R_4(p+1) > 0$ for the $2014$ values of $p \leq 5336500537$ for which $R_3(p) = 0$. For each of these, $R_4(p+1) > 0$.

We believe that the complete list of numbers $n$ for which $R_4(n) = 0$ is:
\begin{eqnarray*}
1, 2, 3, 4, 6, 8, 12, 14, 18, 32, 38, 44, 54,\\
68, 102, 108, 182, 192, 194, 224, 252, 374, 422,\\
432, 908, 1092, 1202, 1278, 2468, 2768,\\
3182, 4508, 7208, 16104, 21998, 26348, 45752
\end{eqnarray*}

\subsection{Residue Classes}

The expression $f_3(x, y, z) = xyz + x + y + z$ is also equal to $$z(xy+1) + x + y.$$
As a result, every positive integer $n \equiv x+y \pmod{xy+1}$ with $n> xy+1$ has $R_3(n) > 0$. For example, if $x = 2$ and $y = 2$, we see that every $n > 5$ such that $n \equiv 4 \pmod{5}$ has $R_3(n) > 0$. Similarly, if $n \equiv 5 \pmod{7}$, if $n \equiv 7 \pmod{11}$ or if $n \equiv 7 \text{ or } 8 \pmod{13}$, then $R_3(n) > 0$. In general, there are
\begin{equation}\label{eqn:fpdfn} f_3(p) = \frac{1}{2}(\tau(p-1) - 2)\footnote{$f_3(p)$ and $f_3(x,y,z)$ denote different functions. The same is true for $f_4(p)$ and $f_4(x, y, z, w).$}\end{equation}
residue classes modulo any odd prime $p$ such that if $n$ is congruent to one of these residue classes mod $p$, then $R_3(n) > 0$.

Consequently, to give a bound on the number of $n \leq N$ for which $R_3(n) = 0$, we can count how many $n \leq N$ are not in any of these residue classes. This idea suggests
using the large sieve to estimate $U_3(n)$, which we will do in the proof of Theorem \ref{thm:lsbound}.

Now, we do the same thing for $f_4(x, y, z, w) = xyzw + x + y + z + w$. Note that
$$xyzw + x + y + z + w = w(xyz + 1) + x + y + z.$$
Just as for $f_3(x,y,z)$, this shows every positive integer $n \equiv x + y + z \pmod{xyz + 1}$ with $n > xyz + 1$ has $R_4(n) > 0$. For example, if $x = 2$, $y = 2$, and $z = 3$, we see that every $n > 9$ such that $n \equiv 7 \pmod{13}$ has $R_4(n) > 0$. In general, there are
\begin{equation}\label{eqn:fp4dfn} f_4(p) = \frac{1}{6}(\tau_3(p-1) - 3)\end{equation}
residue classes modulo any odd prime $p$ such that if $n$ is congruent to one of these residue classes mod $p$, then $R_4(n) > 0$. We use this again in the proof of Theorem \ref{thm:r4thm}.

\section{Proofs}

\subsection{Proofs of Theorems 1 and 4}

\begin{proof}
Observe that we can split the sum into two parts:
$$\sum_{n\leq N}R_3(n) = \sum_{xyz+x+y+z \leq N} 1 = \sum_{xyz+x+y+z \leq \frac{N}{\log N}}1 + \sum_{\frac{N}{\log N} < xyz+x+y+z < N}1$$
For the second part,
$$\sum_{\frac{N}{\log N} < xyz+x+y+z < N}1 = \sum_{xy < N} \sum_{\frac{\frac{N}{\log N} - x - y}{xy+1} < z < \frac{N-x-y}{xy+1}}1$$
If $x+y+z \gg \frac{N}{\log N}$ and $xyz < N$, then $\max(x,y,z) \gg \frac{N}{\log N}$. Without loss of generality, let $z$ be the maximum of $x,y,z$. Then,
$$N \gg xyz \gg xy ~\frac{N}{\log N}.$$
Therefore,
$$x+y \ll xy \ll \log N.$$
Thus,
$$\sum_{xy < N} \sum_{\frac{\frac{N}{\log N} - x - y}{xy+1} < z < \frac{N-x-y}{xy+1}}1 \ll \sum_{\substack{xyz \leq 2N \\ xy \ll \log N}}1 \ll \sum_{xy \ll \log N} \frac{N}{xy} \ll N(\log \log N)^2$$
Therefore, the terms with $x+y+z \gg \frac{N}{\log N}$ are negligible, and we can assume that $x+y+z \ll \frac{N}{\log N}$. Then, we have
$$\sum_{xyz+x+y+z < N}1 = \sum_{xyz < N + O(\frac{N}{\log N})}1 = \sum_{n < N + O(\frac{N}{\log N})} \tau_3(n) \sim \frac{N\log^2 N}{2}$$
(see 12.1.4 of [T]), which proves Theorem 1.
\end{proof}
The proof of Theorem 4 is similar.

\subsection{Proof of Theorem 2}

\begin{proof}
If $n = xyz + x + y + z$, then $nx - x^2 + 1 = (xy+1)(xz + 1)$. Therefore, $$R_3(n) \leq 6\sum_{x \leq n^{\frac{1}{3}}} \tau(nx - x^2 + 1).$$
We now use the following estimate, which is a consequence of Henriot ([H1] and [H2]):
\begin{equation*}
\label{eqn:henriot}
    \sum_{x < n^{\frac13}} \tau(nx - x^2 + 1) \ll n^{\frac13} \log n\prod_{p|n^2+4} (1+4/p).
\end{equation*}
Therefore,
$$R_3(n) \ll n^{\frac13} \log n\prod_{p|n^2+4} (1+1/p)^4.$$
By the prime number theorem,
$$\prod_{p | n^2 + 4} (1 + 1/p) \leq \prod_{p \leq 3\log n} (1 + 1/p).$$
And by Merten's Theorem,
$$\prod_{p \leq 3\log n} (1 + 1/p) \ll \log \log n$$
and the result follows.
\end{proof}

\subsection{Proof of Theorem 3}
\begin{proof}
Define
\begin{equation*}
    a_n = \left\{
    \begin{array}{cc}
         0 & \mbox{if $R_3(n) > 0$} \\
         1 & \mbox{if $R_3(n) = 0$}
    \end{array}
    \right.
\end{equation*}
and define $Z = \sum\limits_{n = 1}^{N} a_n$. Note that $Z = U_3(N)$. By the discussion in section 1.4 and the large sieve inequality (see Theorem 7.11 of [O]), we have that
\begin{equation}
\label{eqn:sieve}
    Z \leq \frac{(N^{\frac{1}{2}} + X)^2}{Q}
\end{equation}
where $X$ is a free parameter and
$$Q = \sum_{q \leq X} \mu^2(q) \prod_{p | q} \frac{f_3(p)}{p - f_3(p)},$$
where $\mu$ denotes the M\"{o}bius function and $f_3(p)$ is defined as in (\ref{eqn:fpdfn}). We now prove the following lemma:
 \begin{lemma}
 \label{lem:q}
 Let $$f_3(p)=\frac{\tau(p-1)-2}{2}$$
 for a prime $p>3$.
 Then, there exists a $c>0$ such that as $X\to \infty$, we have
 \begin{eqnarray*}
Q :=\sum_{q\le X\atop
(q,6)=1}  \mu^2(q) \prod_{p\mid q}\frac{f_3(p)}{p-f_3(p)} \gg \exp\bigg(c\sqrt{\log X}\bigg).
 \end{eqnarray*}
 \end{lemma}
 \begin{proof}
 First of all
 $$Q\ge \sum_{q\le X}  \mu^2(q) \prod_{p\mid q}\frac{f_3(p)}{p}.$$
 Let $g_3(p)=\tau(p-1)$. Then for $p>3$ we have $f_3(p)\ge \frac{g_3(p)}{6}$.
 Therefore, for any $k$, by restricting only to squarefree integers $q$ with $k$ distinct prime factors, we have
 $$Q\ge 6^{-k}\prod_{p_1\cdot p_2 \dots \cdot p_k\le X\atop
 3<p_1<p_2<\dots < p_k} \frac{g_3(p_1)\dots g_3(p_k)}{p_1\dots p_k}.$$
 Now the Titchmarsh divisor problem (as found in Theorem 3.9 of [HR]) asserts that  $$\sum_{3<p\le X} \frac{\tau(p-1)}{p} \sim C  \log X$$
 which implies that for any $0<a<b$,
  $$\sum_{X^a<p\le X^b} \frac{\tau(p-1)}{p} \sim C  (b-a)\log X.$$
  Therefore,
 \begin{eqnarray*}
   Q\ge   6^{-k}\sum_{3<p_1< X^{\frac1{k^2}}}\sum_{X^{\frac1{k^2}}<p_2< X^{\frac2{k^2}}}\dots \sum_{X^{\frac{k-1}{k^2}}<p_k< X^{\frac k{k^2}}}\frac{g_3(p_1)\dots g_3(p_k)}{p_1\dots p_k}\gg (\frac{C}{6k^2}\log X)^k
   \end{eqnarray*} for $X\ge 10$ and any $k<\sqrt{\frac{\log X}{\log 3}}$.
  Now we choose $$ k =b\sqrt{\log X}$$ with a sufficiently small $b$ and have
 $$Q\gg   \exp\left(\sqrt{\log X} \left( b\log \frac{C}{6b^2}\right)\right) \gg \exp(c\sqrt{\log X})$$
 for some small $c>0$
 as claimed.
 \end{proof}

Using the result of Lemma \ref{lem:q} in (\ref{eqn:sieve}) and choosing $X = N^{\frac{1}{2}}$, we obtain the bound
$$U_3(N) \ll \frac{N}{e^{c\sqrt{\log N}}}.$$
\end{proof}

\subsection{Proof of Theorem 6}

The proof is similar to the proof of Theorem \ref{thm:lsbound} except that we use
$$\sum_{3<p\le X} \frac{\tau_3(p-1)}{p} \sim C'  \log^2 X$$
and $f_4(p)$ as defined in (\ref{eqn:fp4dfn}). We leave the details to the reader.

\section{Conclusion and Open Questions}

Our proofs for bounds on $U_3(N)$ and $U_4(N)$ are sieve problems with two different numbers of residue classes. $U_3(N)$ has $c\log p$ residue classes per prime $p$ on average that get sieved out and $U_4(N)$ has $c\log^2 p$ residue classes per prime $p$ on average that get sieved out. It is interesting to note that $U_3(N)$ appears to go to infinity and $U_4(N)$ appears to be bounded.

Another problem which bears similarity is expressing $n$ as the sum of three positive integer cubes. Similar to the proof of Theorem \ref{thm:repbound}, one can obtain $r_3(n) \ll n^{\frac13}\log n (\log \log n)^a$ for some $a$. By Hooley's paucity results, one can obtain $r_3(n) \ll n^{\frac13}$. Hardy and Littlewood's [HL] Hypothesis K was that $r_3(n) \ll n^{\epsilon}$. However, in 1936, Mahler found a parametric family of solutions for $n$'s which are perfect twelfth powers and proved that $r_3(n) = \Omega(n^{\frac{1}{12}})$.

Here are some related problems we think are worth exploring further:

\begin{enumerate}
    \item Prove Conjecture \ref{con:r3bound}, that $R_3(n)$ grows slower than $n^{\epsilon}$ for any $\epsilon > 0$.
    \item Prove Conjecture \ref{conj:five}.
    \item Prove that there are infinitely many positive integers $n$ for which $R_3(n) = 0$.
    \item Does $R_4(n)$ go to $\infty$ with $n$?
    \item Give any improvement on the omega result in (\ref{eqn:omega}).
    \item Give omega results for $\tau(n) + \tau(n+1)$ and $\tau(n^2 + 1)$.
\end{enumerate}
We hope that this work stimulates the reader to pursue some of these questions further.

\section{Acknowledgements}

After posting the first version of this paper to the arXiv, we had a number of email exchanges. We would like to thank Olivier Bordelles, James Maynard, Bill McEachen, Tom\'{a}s Oliviera e Silva, Gerald Tenenbaum, and Trevor Wooley for their helpful comments. We would also like to thank Michel Marcus for adding sequence A350535 to the OEIS.

\end{document}